\documentclass[12pt,amsmath,amscd,amsbsy,amssymb]{amsart}
\usepackage{graphicx}

\theoremstyle{plain}
\newtheorem{theorem}{Theorem}[section]
\newtheorem{lemma}[theorem]{Lemma}
\newtheorem{prop}{Proposition}[section]

\newtheorem{rem}[theorem]{Remark}
\newtheorem{defin}[theorem]{Definition}
\usepackage[all]{xy}
\begin{document}

	\author{Natalia Mazurenko}
	\email{mnatali@ukr.net}
	\address{Department of Mathematics and Computer Science, Vasyl Stefanyk Precarpathian National University,
		57 Shevchenka str., Ivano-Frankivsk, 76025, Ukraine}

	\author{Mykhailo Zarichnyi}
	\email{zarichnyi@yahoo.com}
	\address{Institute of Mathematics, Faculty of Exact and Technical Sciences, University of Rzesz\'ow, Pigonia 1, 35-310 Rzesz\'ow, Poland}

	\title[Invariant idempotent $\ast$-measures]{Invariant idempotent $\ast$-measures for generalized iterated function systems}

	\begin{abstract} The notion of $\ast$-measure on a compact Hausdorff space can be defined for arbitrary continuous triangular norm $\ast$. The well-known Hutchinson-Barnsley theory deals with the iterated function systems (IFSs) of probability measures and establishes existence and uniqueness of invariant measures. 
		
		In the previous paper, IFSs of  $\ast$-measures were considered. In the present paper we deal with  generalized invariant function systems (GIFSs) of $\ast$-measures, which are counterparts of GIFSs in the sense of Mihail and Miculescu. The notion of invariant $\ast$-measure is introduced for such  GIFSs and we prove existence and uniqueness of such elements.

	\end{abstract}

	%\shortAuthorsList{N. Mazurenko}
	%\udc{515.12}

	\maketitle

	\section{Introduction}
	
	The notion of $\ast$-measure, for any triangular norm $\ast$ (see the definition below), is introduced in \cite{Su}. The functor $M^\ast$ of $\ast$-measures acting in the category $\mathbf{Comp}$ of compact Hausdorff spaces and continuous mappings turns out to be isomorphic to the functor of idempotent measures defined in \cite{Z} (see \cite{MSZ}). Idempotent measures are analogues of probability measures in idempotent mathematics (this is the name of the branch of mathematics where the role of the field of real numbers is played by idempotent semi-rings; see, e.g., \cite{KM}). For idempotent measures, one can establish analogues of various results valid for probability measures. In particular, one can prove the existence of invariant idempotent measures for iterated function systems  (IFSs) in full accordance with Hutchinson's results for probability measures; see \cite{MZ}.

	It is also proved in \cite{Su} that the functor $M^\ast$ is the functorial part of a monad in the category $\mathbf{Comp}$. The monad structure is a natural framework for defining the notion of invariant $\ast$-measure for IFSs in metric spaces. The existence and uniqueness of invariant  $\ast$-measures for some IFSs is proved in \cite{MSZ}.
	
	In \cite{MM}, Mihail and Miculescu extended the notion of iterated function system  by considering maps from $X^m$ to $X$, for  metric spaces $X$. For such generalized IFSs (GIFSs) they defined the notion of invariant element and proved, under certain natural conditions, its existence and uniqueness.

	In \cite{Str2}, the probabilistic GIFSs were investigated. It was proved therein that there exists an invariant Hutchinson measure for any such GIFS, i.e., a measure $\mu\in P(X)$ satisfying
	$$\mu=\sum_{i=1}^np_iP(g_i)(\mu\otimes\dots\otimes\mu)$$ 
	(by $P$ we denote the functor of probability measures).
	
	The idempotent counterpart of the probabilistic GIFSs is considered in \cite{COS1}. It  is remarked therein that the used technique can be adopted for idempotent GIFSs and, in particular, the existence and uniqueness theorem for invariant idempotent measures can be proved for these GIFSs.
	
In the present paper we extend results from \cite{MSZ} over GIFSs of $\ast$-measures. Our main result is the existence and uniqueness of invariant $\ast$-measure. Moreover, we  modify the notion of GIFS by considering the maps from the symmetric powers of a space $X$ into $X$.

	There exists a tight connection between our results and results from \cite{OS}. Actually, it is proved in \cite{OS} that there is a homeomorphism between the space $M^\ast(X)$ and the subspace $\bar M(X)$ of the hyperspace of $X\times[0,1]$ consisting of set satisfying some natural properties. It turns out that these sets are exactly the hypographs of some normal upper semicontinuous functions (see below). The latter can be regarded as (upper semicontinuous) fuzzy sets on $X$. Note that an extension of the notion of invariant set on the fuzzy set theory is considered in \cite{CFMV}. The connections of idempotent measures with fuzzy sets are revealed in \cite{OS}, at the same time allowing us to describe $\ast$-measures of fuzzy sets with a grey scale and to derive (some of) our results from the results of the article  \cite{OS}.
	
	Let us note some features of our approach:
	
	1) Our proof of the existence and uniqueness of invariant $\ast$-measures does not rely on metrization of the space of measures and is relatively simple.
	
	2) In generalizations of IFSs and we go even further than \cite{MM} and consider, instead of mappings $X^m\to X$,  mappings of the form $F(X)\to X$, where $F$ is some functorial construction in the category of compact sets. A wide class of examples of such constructions is given by the theory of normal and similar functors of finite degree (see, e.g., \cite{TZ}).
	
	3) We consistently use the formalism of category theory, in particular the notion of monad. This allows us to see commonality in the notions of GIFSs for compact sets, probability measures, idempotent measures, $\ast$-measures, and to develop more or less general definition of invariant object (attractor), which allows also to consider infinite (compact) families of functions.

	\section{Preliminaries}
	
	We assume that the reader is familiar with the language of the category theory. Necessary information can be found, e.g., in \cite{BW}. By ``functor'' we mean ``covariant functor''. Recall that, given categories $\mathcal C,\mathcal D$ and functors $F_1,F_2\colon \mathcal C\to\mathcal D$, a natural transformation $\alpha$ of $F_1$ to $F_2$ (written $\alpha\colon F_1\to F_2$) consists of morphisms $\alpha_X\colon F_1(X)\to F_2(X)$, given for every object $X$ of $\mathcal C$ such that, for any morphism $f\colon X\to Y$ in $\mathcal C$, the diagram
	$$\xymatrix{F_1(X)\ar[r]^{F_1(f)}\ar[d]_{\alpha_X}&F_1(Y) \ar[d]^{\alpha_Y}\\ 
	F_2(X)\ar[r]_{F_2(f)}&F_2(Y)}$$ 
	is commutative. 
	
	Let $\mathcal C$ be a category. Given a functor $T\colon \mathcal C\to\mathcal C$ and $n\in\mathbb N$, we denote by $T^n$ the $n$-th iteration of $T$, $T^2=TT$, $T^3=TT^2=TTT$, \dots
	
	A monad on $\mathcal C$ is a triple $\mathbb T=(T,\eta,\mu)$ which consists of a functor $T\colon \mathcal C\to \mathcal C$ and natural transformations $\eta\colon 1_{\mathcal C}\to T$, $\mu\colon T^2\to T$ such that the diagrams
	
	$$\xymatrix{T\ar[r]^{\eta_T} \ar[rd]_{1}&T^2\ar[d]^\mu & T\ar[l]_{T(\eta)}\ar[ld]^1\\ &T&} \text{ and } \xymatrix{T^3\ar[r]^{T(\mu)}\ar[d]_{\mu_T}& T^2\ar[d]^{\mu}\\ T^2\ar[r]_{\mu} & T}$$
	are commutative (see, e.g., \cite{BW}).
	
	\subsection{$\ast$-measures}
	
	A triangular norm (t-norm) is a continuous, associative, commutative and monotonic binary operation on the unit segment $\mathbb I=[0,1]$ for which 1 is a unit  (see, e.g., \cite{KMP}). The following are examples of triangular norms:
	\begin{enumerate}
		\item (multiplication) $(a,b)\mapsto a\cdot b$;
		\item (minimum) $(a,b)\mapsto \min\{a,b\}=a\wedge b$;
		\item (\L ukasiewicz t-norm) $(a,b)\mapsto \max\{a+b-1,0\} =a\diamond b$.
	\end{enumerate}

	Let $X$ be a compact Hausdorff space. By $C(X,\mathbb I)$ we denote the set of all continuous functions from $X$ to $\mathbb I$. Given $c\in \mathbb I$, we denote by $c_X\in C(X,\mathbb I)$ the constant function with the value $c$. A functional $\mu\colon C(X,\mathbb I)\to\mathbb I$ is called an idempotent $\ast$-measure \cite{Su} if
	
	\begin{enumerate}
		\item $\mu(c_X)=c$;
		\item $\mu (\lambda\ast \varphi)=\lambda\ast \mu (\varphi)$;
		\item $\mu (\varphi\vee \psi)=\mu(\varphi)\vee\mu(\psi)$,
	\end{enumerate}
	for all $c\in\mathbb I$ and $\varphi,\psi\in C(X,\mathbb I)$. (Hereafter $\vee$ denotes the maximum.)
	
	Let $M^\ast(X)$ denote the set of all $\ast$-measures on $X$. 
	
	Given $\mu\in M^\ast(X)$, $\varphi_1,\dots,\varphi_n\in C(X,\mathbb I)$ and $\varepsilon>0$, we define  $$O\langle\mu;\varphi_1,\dots,\varphi_n;\varepsilon\rangle=\{\nu\in M^\ast(X)\mid |\mu(\varphi_i)-\nu(\varphi_i)|<\varepsilon,\ i=1,\dots,n\}.$$

	Then the sets of the form $O\langle\mu;\varphi_1,\dots,\varphi_n;\varepsilon\rangle$, where $\mu\in M^\ast(X)$, $\varphi_1,\dots,\varphi_n\in C(X,\mathbb I)$, $n\in \mathbb N$ and $\varepsilon>0$, comprise a base of the weak* topology on $M^\ast(X)$.
	
	Note that $M^\ast$ is a covariant endofunctor in the category of compact Hausdorff spaces. Given a continuous map $f\colon X\to Y$ of compact Hausdorff spaces, define $M^\ast(f)\colon M^\ast(X)\to  M^\ast(Y)$ by the formula $M^\ast(f)(\mu)(\varphi)=\mu(\varphi f)$, for any  $\mu \in M^\ast(X)$ and $\varphi\in   C(X,\mathbb I)$.

	The functor $M^\ast$ is the functorial part of a monad in the category of compact Hausdorff spaces. See, e.g., \cite{BW} for the definition of monad. The natural transformations  $\eta^\ast\colon 1_{\mathbf{Comp}}\to M^\ast$ by the condition $\eta^\ast_X(x)=\delta_x\in M^\ast(X)$, $x\in X$.

	Given $\varphi\in C(X,\mathbb I)$, define $\bar\varphi\colon M^\ast(X)\to\mathbb I$ by the formula: $\bar\varphi(\mu)=\mu(\varphi)$, $\mu\in M^\ast(X)$. First, note that $\bar\varphi\in C(M^\ast(X),\mathbb I)$.
	
	Given $\mathcal M\in M^{\ast2}(X)$, define $\zeta^\ast_X(\mathcal M)\colon C(X,\mathbb I)\to \mathbb I$ by the formula $\zeta^\ast_X(\mathcal M)(\varphi)=\mathcal M(\bar\varphi)$. Note that $\zeta^\ast_X(\mathcal M)\in M^\ast(X)$ (see, e.g., \cite{SZ}).
	
The monad structure allows us to define the notion of (tensor) product of $\ast$-measures as follows. Let $\mu\in M^\ast(X)$, $\nu\in M^\ast(Y)$, where $X,Y$ are compact Hausdorff spaces. For every $x\in X$, let $f_x\colon Y\to X\times Y$ denote the map $y\mapsto (x,y)$. Then, for $y\in Y$, let $g_y\colon X\to M^\ast(X\times Y)$ denote the map $x\mapsto M^\ast (f_x)(\nu)$. Finally, let $$\mu\otimes \nu=\zeta^\ast_{X\times Y}M^\ast(g_y)(\mu).$$

By induction, given $\mu_i\in X_i$, $i=1,\dots,m$, one can define $$\mu_1\otimes\dots\otimes\mu_m=(((\mu_1\otimes\mu_2)\otimes\mu_3)\otimes\dots)$$

	\subsection{Hyperspace representation} Let $X$ be a compact Hausdorff space. By $\exp X$ we denote the set of all nonempty compact subsets in $X$. Given $U_1,\dots,U_n\subset X$, let $$\langle U_1,\dots,U_n\rangle=\{A\in \exp X\mid A\subset \cup_{i=1}^nU_i,\ A\cap U_i\neq\emptyset,\ i=1,\dots,n\}.$$ The family $$\{\langle U_1,\dots,U_n\rangle\mid U_1,\dots,U_n\text{ are open},\ n\in\mathbb N\}$$ is known to be a base of the {\em Vietoris topology} on the set $\exp X$.
	The obtained topological space is called the {\em hyperspace} of $X$.
	
	If $(X,d)$ is a compact metric space, then the set $\exp X$ is endowed with the
	Hausdorff metric $d_H$,
	$$d_H(A,B)=\max \left\{\sup _{x\in A}\inf _{y\in B}d(x,y),\;\sup _{y\in B}\inf _{x\in A}d(x,y)\right\}.$$
	
	The Hausdorff metric is known to generate the Vietoris topology.
	
	In the following lemma from \cite{MSZ}, the sup metric is considered on the product of metric spaces.
	
	\begin{lemma}\label{prod} Let $X,Y$ be compact metric spaces,  $\mathrm{pr}_Y\colon X\times Y\to Y$ be the projections. Let $A,B\in\exp (X\times Y)$. If $\mathrm{pr}_Y(A)=\mathrm{pr}_Y(B)$, then $$d_H(A,B)\le\mathrm{diam}(X).$$
	\end{lemma}

	Denote by $\bar M(X)$ the set of all $A\in\exp (X\times\mathbb I)$ satisfying the following conditions:
	\begin{enumerate}
		\item $A\cap (X\times\{1\})\neq\emptyset$;
		\item $X\times\{0\}\subset A$;
		\item $A$ is saturated, i.e., if $(x,t)\in A$, then $(x,s)\in A$ for every $s\in[0,t]$.
	\end{enumerate}
	
	Let $f\colon X\to Y$ be a map. Define the map $\bar M(f)\colon \bar M(X)\to\bar M(Y)$ by the formula:
	$$\bar M(f)(A)=\exp(f\times 1_{\mathbb I})(A)\cup (Y\times\{0\}).$$

	Actually, $\bar M$ is a functor in the category $\mathbf{Comp}$. It is proved in \cite{Su} that the functors $\bar M$ and $M^\ast$ are isomorphic for all triangular norms $\ast$. Actually, an isomorphism $h^\ast \colon \bar M\to M^\ast$ can be defined by the formula: $$h_X^\ast(A)(\phi)=\vee\{\phi(x)\ast t\mid (x,t)\in A\},\ A\in M^\ast(X),\phi\in C(X,\mathbb I).$$
	
	In the sequel, having in mind $h^\ast$ we use the notations $\bar M^\ast$ instead of $\bar M$. Note that $$ h_X^\ast(A\cup B)=h_X^\ast(A)\vee h_X^\ast(B),\ A,B\in \bar M^\ast(X),$$ and
	$$h_X^\ast(t\bar\ast A)=t\ast h_X^\ast( A),$$
	where $t\bar\ast A=\{(x,s\ast t)\mid (x,s)\in A\}$.
	
	The isomorphism $h^\ast$ of functors is actually an isomorphism of monads $\bar{\mathbb M}^\ast$ generated by $\bar M^\ast$ and $ M^\ast$. The natural transformations $\bar\eta\colon 1\to \bar M^\ast$ and $\bar\zeta\colon \bar M^{\ast2}\to \bar M^\ast$ are defined as follows:
	\begin{align*} \bar\eta_X^\ast(x)=& (X\times\{0\}) \cup(\{x\}\times\mathbb I),\  x\in X, \\
		\bar\zeta_X^\ast(\mathcal A) =&\{(x,t)\in X\times\mathbb I\mid \text{ there exist }A\in\exp (X\times\mathbb I),\ r,s\in\mathbb I\text{ such that } \\
		&(x,r)\in A,\ (A,s)\in\mathcal A, t=r\ast s\},\ \mathcal A\in \bar M^{\ast2}(X) 
		\end{align*}
	(see \cite{MSZ}).

	We denote by $h^\ast\colon M^\ast\to\bar M^\ast$ the isomorphism defined by the condition $$\mu(\phi)=\max\{\phi(x)\ast t\mid (x,t)\in h^\ast_X(\mu)\},$$
	$\phi\in C(X,\mathbb I)$, $\mu\in M^\ast(X)$.
	
	We will use the following notation: $\bar\mu=h^\ast_X(\mu)$, for every $\mu\in M^\ast(X)$. 
	
	We denote $$\bar\mu_1\bar\times\dots\bar\times\bar\mu_n=h^\ast_{X\times\dots\times X}(\mu_1\otimes\dots\otimes\mu_n).$$
	\begin{lemma}  \begin{align*}\bar\mu_1\bar\times\dots\bar\times\bar\mu_n=&\{(x_1,\dots,x_n;t)\in X\times\dots\times X\times\mathbb I\mid\\ & \text{ there exist }t_1,\dots,t_n\in\mathbb I  \text{ such that }t=t_1\ast\dots\ast t_n,\ (x_i,t_i)\in \bar\mu_i,\\ &i=1,\dots,n\}.\end{align*}
	\end{lemma}
	\begin{proof} It is sufficient to consider the case $n=2$, the general case will follow by induction.
		
		Let $\bar \mu\in \bar M^\ast(X)$, $\bar \nu\in \bar M^\ast(Y)$. Given $x\in X$, we obtain $f_x(y)=(x,y)\in X\times Y$ and, for every $\bar \nu\in \bar M^\ast(Y)$, we obtain $$g_{\bar\nu}(x)=\bar M^\ast(f_x)(\bar\nu)=\{(x,y,t_2)\mid (y,t_2)\in\bar\nu\}.$$
		
		Then \begin{align*}\bar M^\ast(g_{\bar\nu})(\bar\mu)=&\{(g_{\bar\nu}(x),t_1)\mid (x,t_1)\in\bar\mu\} \\ =& \{(\{(x,y,t_2)\mid (y,t_2)\in\bar\nu\},t_1)\mid (x,t_1)\in\bar\mu\}\end{align*}
		and, by the definition of $\bar\zeta^\ast$,
		$$\bar\mu\bar\times\bar\nu= \bar\zeta^\ast_{X\times Y}\bar M^\ast(g_{\bar\nu})(\bar\mu)=\{(x,y,t_1\ast t_2)\mid (x,t_1)\in\bar\mu, (y,t_2)\in\bar\nu \}.$$

	\end{proof}

	\subsection{Normal upper semicontinuous fuzzy sets}\label{ss:nusfs}

	There is a natural bijection between the set $\bar M(X)$ and the set of all normal upper semicontinuous (usc) functions considered in \cite{OS}. Recall that a function $u\colon X\to[0,1]$ is called normal if there is $x\in X$ such that $u(x)=1$, and is called usc if, for each $c \in \mathbb I$ the set $u^{-1}([c,1]) = \{x \in
		X | u(x) \ge c\}$ is closed. 
	
	We denote by $\mathcal F_X^{\bullet}$ the set of all normal usc functions.
	
	\begin{rem} In \cite{OS}, the notation $\mathcal F_X^{\ast}$ is used. We changed it to avoid confusion with the notation of t-norm.
		\end{rem} 
	
	The proof of the following statement is immediate.
	
	\begin{prop} The map $s_X\colon \mathcal F_X^{\bullet}\to \bar M(X)$, $s_X(u)=\{(x,t)\in X\times\mathbb I\mid t\le u(x)\}$, is a bijection. 
		\end{prop}
	
  Given a morphism $T\colon X\to Y$ in the category $\mathbf{Comp}$, the map $T\colon \mathcal F_X^{\bullet}\to \mathcal F_Y^{\bullet}$ is defined by the formula $T(u)(y)=\sup\{u(x)\mid x\in u^{-1}(y)\}$ (convention: $\sup\emptyset=0$). Clearly, $1_X(u)=u$, and $(T_2T_1)(u)=T_2(T_1(u))$, for any morphisms $T_1\colon X\to Y$, $T_2\colon Y\to Z$ in  $\mathbf{Comp}$. This means that  $\mathcal F_X^{\bullet}$ can be regarded as a functor from the category $\mathbf{Comp}$ to the category of sets. The class of maps $s_X$, where $X$ is a compact Hausdorff space, comprices a natural transformation of the functor $\bar M$ into the functor $\mathcal F^{\bullet}$.

	\subsection{$G$-symmetric powers} Let $S_m$ denote the group of bijections of the set $\{1,\dots,m\}$. Any subgroup $G$ of $S_m$ acts on the power $X^m$ of a set $X$ by permutation of coordinates. The orbit space of this action is called the $G$-symmetric power of $X$ and is denoted by $SP^m_GX$. The orbit containing $(x_1,\dots,x_m)\in X^m$ will be denoted by $[x_1,\dots,x_m]_G$. The quotient map $(x_1,\dots,x_m)\mapsto[x_1,\dots,x_m]_G\colon X^m\to SP^m_GX $ will be denoted by $\pi_G$.
	
	Let $(X,d)$ be a metric space. The set $SP^m_GX$ will be endowed with the metric $\hat d$, $$\hat d([x_1,\dots,x_m]_G,[y_1,\dots,y_m]_G)=\min_{\sigma\in G}\max_{1\le i\le m}d(x_i,y_{\sigma(i)}).$$
	
	Note that if $G$ is the trivial group, then $SP^m_GX=X^m$. If $H$ is a subgroup of $G$, then by $\pi_{HG}$ we denote the natural map $$[x_1,\dots,x_m]_H\mapsto [x_1,\dots,x_m]_G\colon SP^m_HX\to SP^m_GX.$$
	
	If $X$ is a metric space, then the map  $\pi_{HG}$ is nonexpanding.

	Given $\mu\in M^\ast(X)$, for every subgroup $G$ of $S_m$, we let 
	$$[\mu\otimes\dots\otimes\mu]_G=M^\ast(\pi_G)(\mu\otimes\dots\otimes\mu).$$
	
		Clearly, $$M^\ast(\pi_{HG})([\mu\otimes\dots\otimes\mu]_H)=[\mu\otimes\dots\otimes\mu]_G.$$
	
	\subsection{Generalized Iterated Function Systems} 	 The power $X^m$ of a metric space $X$ is endowed with the maximum metric. A GIFS on $X$ of order $m$ is a finite set of maps $g_i\colon X^m\to X$, $i=1,\dots,n$. 
	
	We will need the following definition \cite{OS}. A map $f \colon X^m \to X$ is a generalized Matkowski contraction of degree $m$, if for some nondecreasing $\varphi\colon [0,\infty) \to  [0,\infty)$ with $\varphi^{(k)}(t)\to0$ for $t >0$ (here $\varphi^{(k)}$ is the $k$-th iteration of $\varphi$), it holds
	$$d(f (x),f (y)) \le \varphi(d(x,y)),\ x,y \in X^m.$$
	Here $\varphi$ is called a witness for $f$.
	
	If $X$ is compact, $f$ is a  generalized Matkowski contraction of degree $m$ if and only if the following condition holds:
	
	for every $t>0$ there exists $\alpha<1$ such that, for every $x,y\in X^m$,
	$$d(x,y)\ge t \implies d(f(x),f(y))\le\alpha d(x,y).$$

	The notion of a generalized  Matkowski contraction can be formulated for the $G$-symmetric powers.
	
	\begin{defin} An $SP^m_G$-generalized  Matkowski contraction is a  map $f \colon SP^m_GX \to X$ such that, for some nondecreasing $\varphi\colon [0,\infty) \to  [0,\infty)$ with $\varphi^{(k)}(t)\to0$ for $t >0$, it holds
		$$d(f (x),f (y)) \le \varphi(d(x,y)),\ x,y \in SP^m_GX.$$
		\end{defin}

	\begin{rem}
		If  $f \colon SP^m_GX \to X$ is an $SP^m_G$-generalized  Matkowski contraction and $H$ is a subgroup of $G$, then $f\pi_{HG}\colon SP^m_HX \to X$ is an $SP^m_G$-generalized  Matkowski contraction. In particular,  $f\pi_{G}\colon X^m \to X$ is a generalized  Matkowski contraction  of degree $m$. This easily follows from the fact that $\pi_{HG}$ and $\pi_G$ are nonexpanding.
	\end{rem}
	
	Define $\Phi \colon\exp X\to\exp X$ by the formula
	$$\Phi(A)=\cup_{i=1}^nf_i(A^m),\ A\in \exp X.$$

	In the language of category theory this formula has an interpretation analogous to the formulas used to define invariant sets, probability measures, and idempotent measures (see examples below).
	
	Actually, we can introduce the following general notion.
	
	\begin{defin}\label{d:TIFS}
			Let $(T,\eta,\zeta)$ be a monad in the category $\mathbf{Comp}$. An IFS for $\mathbb T$ is a collection of selfmaps $f_i\colon X\to X$, $i=1,\dots,n$, together with an element $a\in T(\{1,\dots,n\})$. Given $m\in T(X)$, we define $\bar f_m\colon \{1,\dots,n\}\to T(X)$ by the formula $\bar f_m(i)=T(f_i)(m)$. Finally, define $\Psi_T\colon T(X)\to T(X)$ by the formula 
		\begin{equation}\Psi_T(m)=\zeta_XT(\bar f_m)(a),\end{equation}
		$m\in T(X)$. 
	\end{defin}

	The following are examples.
	
	1) The hyperspace functor $\exp$ determines the hyperspace monad $(\exp,s,u)$, where, for every $X$, $s_X(x)=\{x\}$, $x\in X$, and $u_X\colon \exp^2X\to\exp X$ is the union map.   Choosing $a=\{1,\dots,n\} \in \exp \{1,\dots,n\}$ we obtain the Hutchinson map \cite{Hu}
	$$A\mapsto\Psi_{\exp}(A)=\cup_{i=1}^n\exp f_i(A)= \cup_{i=1}^n f_i(A), $$
	$A\in\exp X$.
	
	2) The probability measure functor $P$ determines the monad $(P,\delta,\zeta)$, where $\delta_X(x)=\delta_x$ (the Dirac measure concentrated at $x\in X$). If $a\in P(\{1,\dots,n\})$, then $a=\sum_{i=1}^n\alpha_i\delta_i$ and we obtain the map 	$$\mu\mapsto \Psi_{P}(\mu)=\cup_{i=1}^n\exp f_i(A)= \cup_{i=1}^n f_i(A), $$
	$A\in\exp X$ (see \cite{Hu}).
	
	3) The idempotent measure functor $I$ determines the monad $(I,\eta,\zeta)$ on the category $\mathbf{Comp}$ (see \cite{Z} for a description). The corresponding IFSs and their attractors (invariant idempotent measures) are considered in \cite{MZ}.

	Having in mind these examples one can formulate the natural counterpart of the notion of  $SP^m_G$-GIFS of $\ast$-measures.
\begin{defin}\label{d:inv}
	Let $G$ be a subgroup of the symmetric group $S_m$. An  $SP^m_G$-GIFS is a collection of maps $f_i\colon SP^m_G(X)\to X$, $i=1,\dots,m$, with an $\ast$-measure $\vee_{i=1}^m\alpha_i\ast\delta_{i}$.   
	
	Define $\Psi\colon M^\ast (SP^m_G(X))\to M^\ast(X)$ by the formula \begin{equation}\label{f:2}
		\Psi(\nu)=\vee_{i=1}^n\alpha_i\ast M^\ast(f_i)(\nu).\end{equation}
	
	An element $\mu\in M^\ast (SP^m_G(X))$ is called invariant if 
	$$\mu=\Psi([\mu\otimes\dots\otimes\mu]_G).$$

\end{defin}  

\begin{rem}\label{r:GH} Let $G,H$ be subgroups of the symmetric group $S_m$ and $H\subset G$. Every $G$-GIFS $f_i$, $i=1,\dots,m$, of order $m$ naturally generates an $H$-GIFS of order $m$, namely $f_i\pi_{HG}$, $i=1,\dots,m$, and any invariant element $\mu$ of the former is the invariant element of the latter. Indeed, \begin{align*}
		&\vee_{i=1}^n\alpha_i\ast M^\ast(f_i\pi_{HG})([\mu\otimes\dots\otimes\mu]_H)\\
	=	&\vee_{i=1}^n\alpha_i\ast M^\ast(f_i) M^\ast(f_i)(\pi_{HG})([\mu\otimes\dots\otimes\mu]_H)\\
	=& 	\vee_{i=1}^n\alpha_i\ast M^\ast(f_i)([\mu\otimes\dots\otimes\mu]_G)=\mu.
	\end{align*}

	\end{rem}

\begin{rem}\label{r:TIFS}
	Definition \ref{d:TIFS} can be extended over some infinite collections of functions.  For any compact metric space $X$, denote by $C(X,X)$ the set of continuous selfmaps of $X$ endowed with the topology of uniform convergence.  Let $K$ be a compact metric space and $g\colon K\to C(X,X)$ be a continuous map. 
	
	Let $(T,\eta,\zeta)$ be a monad in the category $\mathbf{Comp}$. We assume that the functor $T$ is continuous on morphisms, i.e., the map $T\colon C(X,X)\to C(T(X),T(X))$ is continuous in the topology of uniform convergence. The mentioned functors as well as many other functors satisfy this condition \cite{TZ}.

	 Let $a\in T(K)$. We are going to define a map $\Psi_{K,g,a}\colon T(X)\to T(X)$ as follows. Let $m\in T(X)$.  First, define $h\colon K\to T(X)$ by the formula $h(x)=T(g(x))(m)$, $x\in K$. It easily follows from the continuity of $T$ that $h$ is continuous. Finally, we define $$\Psi_{K,g,a}(m)=\zeta_XT(h)(a).$$

	An element $m\in T(X)$ is called invariant if $\Psi_{K,g,a}(m)=m$.

\end{rem}

For the case of simplicity, we deal with the IFSs in the  considerations from Remark \ref{r:TIFS}; the generalization is immediate as there is the definition of tensor product for monads in the category $\mathbf{Comp}$ (see \cite{TZ}).

\begin{rem}
	In the language of normal upper semicontinuous fuzzy sets  see \ref{ss:nusfs} above), the map $t\mapsto\alpha\ast t\colon\mathbb I\to\mathbb I$ is a grey level map in the sense of \cite{OS}. This implies that Definition \ref{d:inv} is a partial case of  \cite[Definition 3.11]{OS} in the case of trivial group $G$.
\end{rem}

\section{Results}
	
	\begin{theorem}\label{t:main} Let $(X,d)$ be a compact metric space and let $g_i\colon SP^m_GX\to X$, $i=1,\dots,n$, be  $SP^m_G$-generalized Matkowski contractions, where $G$ is a subgroup of $S_m$. Let $\ast$ be a triangular norm and $\vee_{i=1}^n\alpha_i\ast\delta_i\in M^\ast(\{1,\dots,n\})$.
		
		There exists a unique invariant $\ast$-measure for the obtained $SP^m_G$-GIFS.
\end{theorem}

\begin{proof} Since, for any subgroup $H$ of $G$, the map $\pi_{HG}$ is nonexpanding, the maps $g_i\pi_{HG}$, $i=1,\dots,n$, are also Matkowski contractions. Having in mind Remark \ref{r:GH}, we may therefore assume that $H=\{e\}$ and deal with maps $f_i=g_i\pi_{\{e\}G}$ from $X^m=SP^m_{\{e\}}$ to $X$. 
		
		Given $\mu\in M^\ast(X)$, define a sequence $\{\mu_i\}_{i=0}^\infty$ as follows: $\mu_0=\mu$ and $\mu_i=\Psi(\mu_{i-1})$, for all $i\ge1$.
	
The Hutchinson map $\Phi\colon \exp X\to\exp X$ is defined by the formula $\Phi(A)=\cup_{i=1}^nf_i(A)$, $A\in\exp X$.

\medskip

	 \textbf{Claim.} There exists $\mu\in M^\ast(X)$ such that $\bar\mu_i\supset\bar\mu_{i-1}$, for all $i\in\mathbb N$.
	 
	 \medskip
	
Indeed, let $A$ be the invariant set for $\Phi$. We proceed by induction. Let $\bar\mu=(X\times\{0\})\cup (A\times\mathbb I)$. Then, clearly, $\bar\mu_0=\bar\mu\supset \bar \mu_1$.
		
		Suppose that $\bar\mu_i\supset\bar\mu_{i-1}$, for all $i\le k$. Then $$\bar\mu_{k+1}=\bigcup_{i=1}^n\alpha_i\ast M^\ast(g_i)(\bar\mu_k\bar\times\dots\bar\times\bar\mu_k)\supset \bigcup_{i=1}^n\alpha_i\ast M^\ast(g_i)(\bar\mu_{k-1}\bar\times\dots\bar\times\bar\mu_{k-1})=\bar\mu_k.$$
	
	\medskip
	
	 Now, existence follows from the Claim. Indeed, given $\mu_0\in M^\ast (X)$ satisfying $\Psi(\bar\mu_i)\supset \bar\mu_{i-1}$ for all $i\in \mathbb N$, we see that $\nu=\lim_{i\to\infty}\mu_i$ is a required $\ast$-measure as $$\bar\Psi (\bar\nu)=\bar\Psi\left(\lim_{i\to\infty}\bar\mu_i\right)=\bar\Psi\left(\cap_{i=0}^\infty\bar\mu_i\right)=\cap_{i=0}^\infty\bar\Psi(\bar\mu_i)=\cap_{i=1}^\infty\bar\mu_i=\bar\nu.$$
		
		Uniqueness. Suppose that $\bar \mu$ and $\bar\nu$ are two invariant $\ast$-measures and $d(\bar\mu,\bar\nu)=C>0$. There exists $\alpha<1$ such that:
		$$d(x,y)\ge C/2 \implies d(f_i(x),f_i(y))\le\alpha d(x,y)$$ for all $i=1,\dots,n$. There exists $k\in\mathbb N$ such that $\alpha^k\mathrm{diam}\,X\le C/2$. 
		
		Let $X_0=X$, $X_1=\Phi(X_0)$, and $X_{i+1}=\Phi(X_i)$ for all $i\ge0$.

		We are going to show that $\mathrm{diam}(X_k)\le C/2$. Indeed, assume the  contrary, then there exist $x^{(k)}, y^{(k)}\in X_k$ such that $d(x^{(k)}, y^{(k)})> C/2$. There exist $\tilde x^{(k-1)}, \tilde y^{(k-1)}\in X_{k-1}^m$ such that $x^{(k)}=f_{i}(\tilde x^{(k-1)})$, $y^{(k)}=f_{i}(\tilde y^{(k-1)})$, for some $i$, and, since the generalized Matkowski contractions are contractions, $$C/2\le d(x^{(k)},y^{(k)})\le d(\tilde x^{(k-1)},\tilde y^{(k-1)}).$$ Therefore  $d(x^{(k)},y^{(k)})\le \alpha d(\tilde x^{(k-1)},\tilde y^{(k-1)}).$ By the definition of the metric on $X^m$, there exist $x^{(k-1)},y^{(k-1)}\in X_{k-1}$ such that $d(\tilde x^{(k-1)},\tilde y^{(k-1)})=
		d(\tilde x^{(k-1)},\tilde y^{(k-1)})$.	Thus,  $d(x^{(k)},y^{(k)})\le \alpha d( x^{(k-1)}, y^{(k-1)}).$
		
		Proceeding as above, one can find recursively $x^{(i)},y^{(i)}\in X_i$, $i=k-2,\dots,1,0$, such that $d(x^{(i+1)},y^{(i+1)})\le \alpha d( x^{(i)}, y^{(i)})$ for all $i=k-2,\dots,1,0$. We then obtain $$C/2<d(x^{(k)}, y^{(k)})\le \alpha^k d(x^{(0)}, y^{(0)}) \le \alpha^k \mathrm{diam}(X)\le C/2 $$ and we obtain a contradiction.
		
		\end{proof}
	
	\section{Remarks and open questions}
	
	 Thus, the class of invariant elements of  $G$-GIFS of order $m$ is contained in the class of invariant elements of  $H$-GIFS of order $m$. We conjecture that, in general, these classes do not coincide. 
	
	The considerations of $G$-GIFSs can be extended in two directions. First, instead of $\ast$-measures one can consider compact sets, probability measures etc. Second, one can replace the $G$-symmetric powers by another constructions known as functors of finite degree. As an example, we consider the $m$-th hypersymmetric powers $\exp_mX$, i.e. the spaces of nonempty subsets of $X$ of cardinality $\le m$ endowed with the Vietoris topology. An $\exp_m$-GIFS is then a family of maps $g_i\colon \exp_mX\to X$,  $i=1,\dots,n$, and a compact subset $A$ of $X$ is called invariant if $A=\cup_{i=1}^n\exp_mA$. Many examples of such constructions can be found in  the general theory of functors of finite degree (see, e.g., \cite{TZ}).

	The notion of place-dependent
	idempotent iterated function systems and its invariant idempotent measure considered in \cite{MO} has its natural analog in the class of $\ast$-measures.  We return to these IFSs in a subsequent publication.

\end{document}